\documentclass[10pt]{article}

\usepackage{amsmath}
\usepackage{amsthm}
\usepackage{amssymb}
\usepackage{latexsym}
\usepackage{pstricks}
\usepackage{graphicx, pst-plot, pst-node, pst-text, pst-tree}
\usepackage{titlefoot}
\usepackage{enumerate}
\usepackage{titlesec}

\usepackage{cancel}
\usepackage{rotating} 
\usepackage[small,it]{caption}
\usepackage{color}
\definecolor{lightgray}{rgb}{0.8, 0.8, 0.8}
\definecolor{darkgray}{rgb}{0.7, 0.7, 0.7}
\definecolor{darkblue}{rgb}{0, 0, .4}

\usepackage[bookmarks]{hyperref}
\hypersetup{
        colorlinks=true,
        linkcolor=darkblue,
        anchorcolor=darkblue,
        citecolor=darkblue,
        urlcolor=darkblue,
        pdfpagemode=UseThumbs,
        pdftitle={Passing through a stack $k$ times},
        pdfsubject={Combinatorics},
        pdfauthor={Howard Skogman and Rebecca Smith },
}

\newtheorem{theorem}{Theorem}[section]
\newtheorem{proposition}[theorem]{Proposition}
\newtheorem{lemma}[theorem]{Lemma}
\newtheorem{definition}[theorem]{Definition}
\newtheorem{corollary}[theorem]{Corollary}

\newtheorem{notation}[theorem]{Notation}

\newtheoremstyle{example}{\topsep}{\topsep}%
     {}
     {}
     {\bfseries}
     {.}
     {.5em}
     {\thmname{#1}\thmnumber{ #2}}
\theoremstyle{example}
\newtheorem{example}[theorem]{Example}

\newtheoremstyle{negexample}{\topsep}{\topsep}%
     {}
     {}
     {\bfseries}
     {.}
     {.5em}
     {\thmname{#1}\thmnumber{ #2}}
\theoremstyle{negexample}

\newcounter{todocounter}

\long\def\symbolfootnote[#1]#2{\begingroup%
\def\thefootnote{\fnsymbol{footnote}}\footnote[#1]{#2}\endgroup}

\makeatletter
\newcommand{\Rm}[1]{\expandafter\@slowromancap\romannumeral #1@}
\newfont{\footsc}{cmcsc10 at 8truept}
\newfont{\footbf}{cmbx10 at 8truept}
\newfont{\footrm}{cmr10 at 10truept}
\renewenvironment{abstract}%
                {
                  \begin{list}{}%
                     {\setlength{\rightmargin}{1in}%
                      \setlength{\leftmargin}{1in}}%
                   \item[]\ignorespaces\begin{small}}%
                 {\end{small}\unskip\end{list}}

 \amssubj{05A05, 05A15, 06A07}

\newpagestyle{main}[\small]{
        \headrule
        \sethead[\usepage][][]
        {\sc  Passing through a stack $k$ times }{}{\usepage}}

\title{\sc{Passing through a stack $k$ times}}
\author{Toufik Mansour
\\[-0.25ex]
\small Department of Mathematics\\[-0.5ex]
\small University of Haifa\\[-0.5ex]
\small 3498838 Haifa, Israel\\[15pt]
Howard Skogman
\\[-0.25ex]
\small Department of Mathematics\\[-0.5ex]
\small SUNY Brockport\\[-0.5ex]
\small Brockport, New York\\[15pt]
Rebecca Smith
\\[-0.25ex]
\small Department of Mathematics\\[-0.5ex]
\small SUNY Brockport\\[-0.5ex]
\small Brockport, New York\\[-1.5ex]
}

\titleformat{\section}
        {\large\sc}
        {\thesection.}{1em}{}

\date{}

\begin{document}
\maketitle

\pagestyle{main}

\newcommand{\s}{\mathbf{s}}
\newcommand{\m}{\mathbf{m}}
\renewcommand{\t}{\mathbf{t}}
\renewcommand{\b}{\mathbf{b}}
\newcommand{\f}{\mathbf{f}}
\newcommand{\rev}{\operatorname{rev}}
\newcommand{\dual}{\operatorname{dual}}
\newcommand{\D}{\mathcal{D}}
\newcommand{\Av}{\operatorname{Av}}

\newcommand{\R}{\stackrel{R}{\sim}}
\renewcommand{\L}{\stackrel{L}{\sim}}
\newcommand{\notR}{\stackrel{R}{\not\sim}}
\newcommand{\notL}{\stackrel{L}{\not\sim}}

\newcommand{\OEISlink}[1]{\href{http://oeis.org/#1}{#1}}
\newcommand{\OEISref}{\href{http://oeis.org/}{OEIS}~\cite{sloane:the-on-line-enc:}}
\newcommand{\OEIS}[1]{(Sequence \OEISlink{#1} in the \OEISref.)}

\def\sdwys #1{\xHyphenate#1$\wholeString}
\def\xHyphenate#1#2\wholeString {\if#1$%
\else\say{\ensuremath{#1}}\hspace{2pt}%
\takeTheRest#2\ofTheString
\fi}
\def\takeTheRest#1\ofTheString\fi
{\fi \xHyphenate#1\wholeString}
\def\say#1{\begin{turn}{-90}\ensuremath{#1}\end{turn}}

\newenvironment{onestack}
{
	\begin{scriptsize}
	\psset{xunit=0.0355in, yunit=0.0355in, linewidth=0.02in}
	\begin{pspicture}(0,-2)(32,20)
	\psline{c-c}(0,15)(10,15)
	\psline{c-c}(13,15)(13,2)(19,2)(19,15)
	\psline{c-c}(22,15)(32,15)
	\rput[l](-0.5,12.5){\mbox{output}}
	\rput[r](32,12.5){\mbox{input}}
}
{
	\end{pspicture}
	\end{scriptsize}
}

\newcommand{\fillstack}[3]{%
	\rput[l](-0.5,17.5){\ensuremath{#1}}
	\rput[c](16.1, 7.5){\begin{sideways}{\sdwys{#2}}\end{sideways}}
	\rput[r](32,17.5){\ensuremath{#3}}
}

\newcommand{\stackinput}{%
	\psline[linecolor=darkgray]{c->}(24, 17.5)(16, 17.5)(16, 14)
}
\newcommand{\stackshortinput}{%
	\psline[linecolor=darkgray]{c->}(22.5, 17.5)(16, 17.5)(16, 14)
}
\newcommand{\stackoutput}{%
	\psline[linecolor=darkgray]{c->}(16, 14)(16, 17.5)(10, 17.5)
}
\newcommand{\stackinoutput}{%
	\psline[linecolor=darkgray]{c->}(24, 17.5)(17, 17.5)(17, 13)
	\psline[linecolor=darkgray]{c->}(15,14)(15, 17.5)(10, 17.5)
}

\newcommand{\firstpass}{%
	\rput[c](16, -.5){\text{First pass through the stack.}}
}
\newcommand{\secondpass}{%
	\rput[c](16, -.5){\text{Second pass through the stack.}}
}
\newcommand{\thirdpass}{%
	\rput[c](16, -.5){\text{Third pass through the stack.}}
}

\begin{abstract}
We consider the number of passes a permutation needs to take through a stack if we only pop the appropriate output values and start over with the remaining entries in their original order.  We define a permutation $\pi$ to be $k$-pass sortable if  $\pi$ is sortable using $k$ passes through the stack.  Permutations that are $1$-pass sortable are simply the stack sortable permutations as defined by Knuth.  We define the permutation class of $2$-pass sortable permutations in terms of their basis.  We also show all $k$-pass sortable classes have finite bases by giving bounds on the length of a basis element of the permutation class for any positive integer $k$.  Finally, we define the notion of tier  of a permutation $\pi$ to be the minimum number of passes \emph{after} the first pass required to sort $\pi$. We then give a bijection between the class of permutations of tier $t$ and a collection of integer sequences studied by Parker~\cite{Parker_thesis}.  This gives an exact enumeration of tier $t$ permutations of a given length and thus an exact enumeration for the class of $(t+1)$-pass sortable permutations. Finally, we give a new derivation for the generating function in \cite{Parker_thesis} and an explicit formula for the coefficients.
\end{abstract}

\section{Introduction}

We begin with the notion of permutation (or pattern) containment.

\begin{definition}  A permutation $\pi=\pi_1\pi_2\dots\pi_n\in S_n$ is said to \emph{contain} a permutation $\sigma=\sigma_1\sigma_2\ldots\sigma_k$ if there exist indices $1\le \alpha_1<\alpha_2< \ldots <\alpha_k \le n$ such that $\pi_{\alpha_i} < \pi_{\alpha_j}$ if and only if $\sigma_i <\sigma_j$.  Otherwise, we say $\pi$ \emph{avoids} $\sigma$.
\end{definition}

\begin{example}  The permutation $\pi = 4127356$ contains $231$ since the $4,7,3$ appear in the same relative order as $2,3,1$.   However, $\pi$ avoids $321$ since there is no decreasing subsequence of length three in $\pi$.
\end{example}

A stack is a last-in first-out sorting device that utilizes push and pop operations.  In Volume $1$ of \emph{The Art of Computer Programming}~\cite{knuth:the-art-of-comp:1}, Knuth showed that the permutation $\pi$ can be sorted (meaning that by applying push and pop operations to the sequence of entries $\pi_1,\dots,\pi_n$ one can output the sequence $1,\dots,n$) if and only if $\pi$ avoids the permutation $231$.  Shortly thereafter Tarjan~\cite{tarjan:sorting-using-n:}, Even and Itai~\cite{even:queues-stacks-a}, Pratt~\cite{pratt:computing-permu:}, and Knuth himself in \emph{Volume 3}~\cite{knuth:the-art-of-comp:3} studied sorting machines made up of multiple stacks in series or in parallel.

Classifying the permutations that are sortable by such a machine is one of the key areas of interest in this field.  To better do so, we will use the following definitions.

\begin{definition}
A \emph{permutation class} is a downset of permutations under the containment order.  Every permutation class can be specified by the set of minimal permutations which are \emph{not} in the class called its \emph{basis}.  For a set $B$ of permutations, we denote by $\Av(B)$ the class of permutations which do not contain any element of $B$.
\end{definition}

For example, Knuth's result says that the stack-sortable permutations are precisely $\Av(231)$, that is the basis for the stack sortable permutations is $\{231\}$.  Given most naturally defined sorting machines, the set of sortable permutations forms a class.  This is often because a subpermutation of a sortable permutation can be sorted by ignoring the operations corresponding to absent entries.~\footnote{An exception is West's notion of $2$-stack-sortability~\cite{west:sorting-twice-t:}, which is due to restrictions on how the machine can use its two stacks.  Namely this machine prioritizes keeping large entries from being placed above small entries.  Because of this limitation, this machine can sort $35241$, but not its subpermutation $3241$.}

Given that the class of permutations sortable by a single stack have a basis of only one element, namely $231$, expecting that the sortable permutations for a network made up of more than one stack would also have a finite basis seems reasonable.
However, this is not the case for machines made up $k\ge 2$ stacks in series or in parallel.  That these machines must have infinite bases was shown by Murphy~\cite{murphy:restricted-perm:} and Tarjan~\cite{tarjan:sorting-using-n:}, respectively.  Moreover, the exact enumeration question is unknown; see Albert, Atkinson, and Linton~\cite{albert:permutations-ge:} for the best known bounds. For a general overview of stack sorting, we refer the reader to a survey by B\'ona~\cite{bona:a-survey-of-sta:}.

In part because of the difficulties noted above, numerous researchers have considered weaker machines.  Atkinson, Murphy, and Ru\v{s}kuc~\cite{atkinson:sorting-with-tw:} considered sorting with two \emph{increasing} stacks in series, i.e., two stacks whose entries must be in increasing order when read from top to bottom\footnote{Even without this restriction, the final stack must be increasing if the sorting is to be successful.}.  They characterized the permutations this machine can sort with an infinite list of forbidden patterns, and also found the enumeration of these permutations.  Interestingly, these permutations are in bijection with the $1342$-avoiding permutations previously counted by B\'ona~\cite{bona:exact-enumerati:}.  Similarly, the third author~\cite{smith:a-decreasing-st:} studied a machine where the first stack must have entries in decreasing order when read from top to bottom.  This permutation class of sortable permutations was shown to be $\Av(3241,3142)$, known to be enumerated by the Schr\"oder numbers by Kremer~\cite{kremer:permutations-wi:,kremer:postscript:-per:} and later an explicit bijection was given by Schroeder and the third author~\cite{schroder}.
A different version, sorting with a stack of depth $2$ followed by a standard stack (of infinite depth), was studied by Elder~\cite{elder:permutations-ge:}.  He characterized the sortable permutations with a finite list of forbidden patterns, but did not enumerate these permutations.

We apply a sorting algorithm on a stack whereby the entries of the permutation are pushed into the stack in the usual way.  We remove or pop an entry $\pi_i$ from the stack only if $\pi_i$ are the next needed entry for the output (namely, $\pi_i$ is the next entry of the identity permutation).
That is, we will allow larger entries to be placed above smaller entries, but we will not allow entries to be pushed to the output prematurely.  In particular, this means that if a permutation contains the pattern $231$, then there will be entries left in the stack after all legal moves have been made.  In this case, the algorithm is repeated on the stack working from the bottom of the stack to the top. That is, the remaining entries are returned to the input to be read in the same order they were the first time.

In some respects this sorting algorithm behaves similarly to West's algorithm for $2$-stack-sorting~\cite{west:sorting-twice-t:} where entries were run through a single stack twice.  However, unlike West's algorithm that prioritizes keeping the stack in increasing order, we will prioritize only outputting the proper entries.  The main advantage is that our sortable permutations will form a permutation class, as deleting entries from a permutation will not impede its ability to be sorted by this algorithm.

\begin{definition}  We will call each repetition of the algorithm used to sort a permutation a \emph{pass}.  Further, the \emph{tier} of a permutation $\sigma$ will refer to the minimum number of times we need to restart the sorting; that is, the \emph{tier} is one less than the minimal number of passes necessary to sort $\sigma$. We use $t(\sigma)$ to denote the tier of the permutation $\sigma$.
\end{definition}

\begin{example}
The permutation $231$ has tier $t(231)=1$, all other elements of $S_3$ have tier $0$.
\end{example}

We translate the original stack sortable requirement to the following theorem.

\begin{theorem}~\label{tier_0}  (Knuth) A permutation $\pi$ has tier $0$, that is $\pi$ can be sorted via single pass through the stack, if and only if $\pi$ avoids the pattern $231$.
\end{theorem}

We note that another way to think of our sorting machine when applied to a permutation with tier $t$ is as a network of $t+1$ input-restricted deques in series with a special output condition. Namely, entries of our permutation may only enter the top of the deque and then either exit the top of the deque to go to the output (immediately passing through the other deques if one is so inclined), or exit the bottom of the deque (and enter the top of the next deque) when no more entries are available to enter the deque and the top entry of the deque is not the next entry to be output.

\section{Separable pairs and permutation classes}

In order to investigate the tiers of permutations more generally, we first give an explicit condition on permutations that describes their tier.

\begin{definition} Let $\sigma \in S_m$ and let $i \in \{1, 2, ... ,n - 1\}$. We say that the integers $(i + 1, i)$ are a \emph{separated pair} in a permutation $\sigma \in S_n$ if there is a subsequence in $\sigma$ of the form $(i + 1, k, i)$ where $k > i + 1$.
\end{definition}

Equivalently one could say that $(i + 1, i)$ are a separated pair in $\sigma$ if they occur as part of a $231$ pattern where $i + 1$ is the middle valued number and $i$ is the smallest number in the pattern. We may also say that the element $k$ \emph{separates} $i + 1$ and $i$. The avoidance of a separated pair in a permutation $\sigma$ is in fact equivalent to $\sigma$ avoiding the pattern $231$.  In particular, what we will prove in Proposition~\ref{sep_pair_prop} is equivalent to Lemma 2 (after taking the complement of the inverse of the permutations involved) of a paper by Claesson~\cite{claesson:generalized-pat}.  Claesson was studying what was known then as \emph{generalized pattern} avoidance, introduced by Babson and Steingr{\'{\i}}msson~\cite{babson:generalized-per:}.  Now such patterns are known by the less misleading term, \emph{covincular}.  For a thorough study of such pattern avoidance we refer the reader to the survey~\cite{einar:survey} by Steingr{\'{\i}}msson and book~\cite{kitaev:book} by Kitaev.

\begin{example}  The permutation $\pi = 5412736$ contains two $231$ patterns, namely $573$ and $473$.   However, $\pi$ contains only one separated pair, namely $(4,3)$.\end{example}

\begin{proposition}
\label{sep_pair_prop}
A permutation $\pi$ has tier $t(\pi) >0$, that is $\pi$ cannot be sorted in a single pass through the stack, if and only if $\pi$ contains a separated pair.
\end{proposition}

\begin{proof} We simply prove that containing a $231$ pattern is equivalent to having a separated pair. If the permutation $\pi$ contains a separated pair then the separated pair along with their ``separator'' is a $231$ pattern. Conversely, assume $\pi$ contains a $231$ pattern which we denote $(i, j, k)$ with $k < i < j$. If $k = i - 1$ we are done. Otherwise consider the position of $i - 1$ relative to $j$. If $i - 1$ is on the same side of $j$ as $k$ then $(i, i - 1)$ form a separated pair in $\pi$. Otherwise, $i - 1$ is on the same side of $j$ as $i$ in which case we have another $231$ pattern in $\pi$, $(i - 1, j, k)$. We continue inductively to see $\pi$ must contain a separated pair.
\end{proof}

More generally, we see that the number of separated pairs in a permutation characterizes its tier.

\begin{theorem}~\label{tier_num_pairs}
The tier of a permutation under this sorting algorithm is exactly the number of separated pairs in the permutation.
\end{theorem}

\begin{proof}
Assume the permutation $\pi$ has $t$ separated pairs. We note that if $t = 0$, then we know $\pi$ is sortable from Proposition~\ref{sep_pair_prop} and hence the tier of $\pi$ is zero.

Consider a pass through the stacks where $i$ is pushed to the output, but $i+1$ is not.  The sorting algorithm will need to be restarted in this scenario if and only if $i + 1$ lies in the stack below another number, say $j$ which must necessarily be larger than $i + 1$ since the algorithm has reached $i$. Hence $(i + 1, i)$ was a separated pair in $\pi$. Moreover, $(i + 1,i)$ must be the smallest separated pair in $\pi$ for $i$ to be pushed to the output in this pass.  Restarting then continues with the remaining permutation values larger than $i$.

Therefore each extra pass corresponds to a unique separated pair.
\end{proof}

\begin{example}  The permutation $\pi=356124$ has two separated pairs, $(3,2),(5,4)$ and thus has tier $2$.  We show the sorting of $\pi$ using three passes through a stack in Figure~\ref{fig_356124}.
\end{example}

\begin{figure}[t]
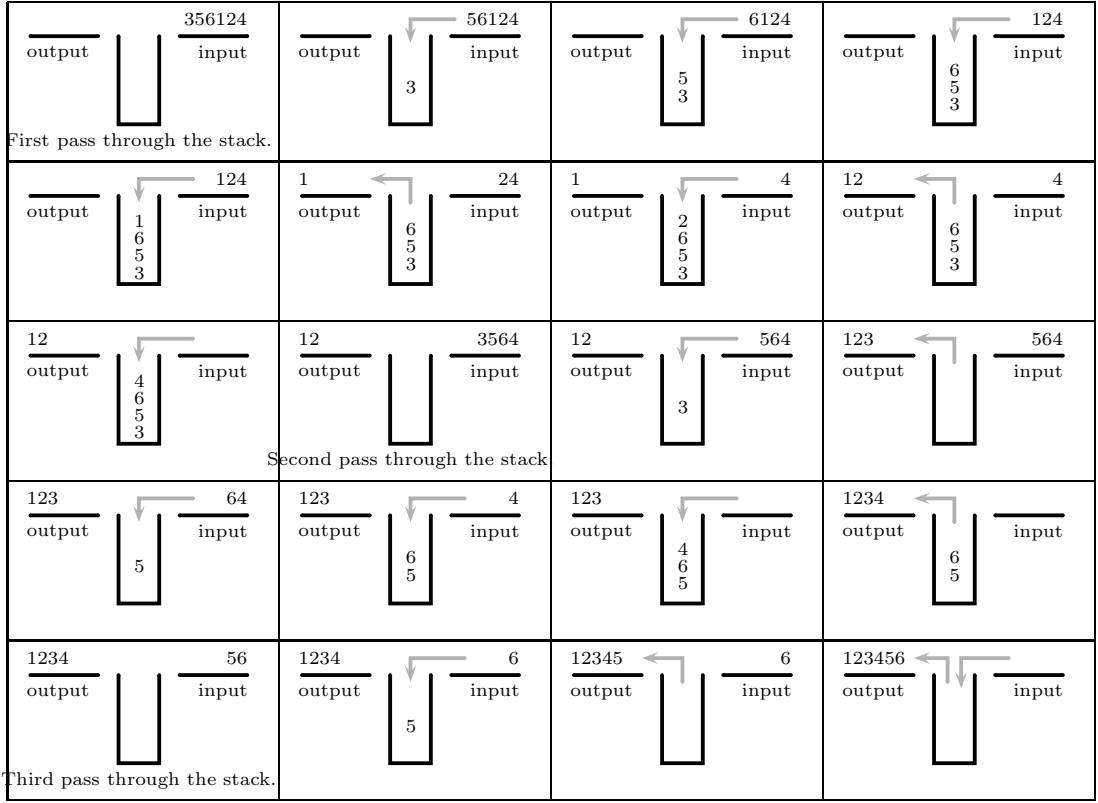

\begin{center}

\begin{tabular}{|c|c|c|c|}
\hline
\begin{onestack}
\fillstack{}{}{356124}
\firstpass
\end{onestack}
&
\begin{onestack}
\fillstack{}{3}{56124}
\stackshortinput
\end{onestack}
&
\begin{onestack}
\fillstack{}{35}{6124}
\stackinput
\end{onestack}
&
\begin{onestack}
\fillstack{}{356}{124}
\stackinput
\end{onestack}
\\\hline
\begin{onestack}
\fillstack{}{3561}{124}
\stackinput
\end{onestack}
&
\begin{onestack}
\fillstack{1}{356}{24}
\stackoutput
\end{onestack}
&
\begin{onestack}
\fillstack{1}{3562}{4}
\stackinput
\end{onestack}
&
\begin{onestack}
\fillstack{12}{356}{4}
\stackoutput
\end{onestack}
\\\hline
\begin{onestack}
\fillstack{12}{3564}{}
\stackinput
\end{onestack}
&
\begin{onestack}
\fillstack{12}{}{3564}
\secondpass
\end{onestack}
&
\begin{onestack}
\fillstack{12}{3}{564}
\stackinput
\end{onestack}
&
\begin{onestack}
\fillstack{123}{}{564}
\stackoutput
\end{onestack}
\\\hline
\begin{onestack}
\fillstack{123}{5}{64}
\stackinput
\end{onestack}
&
\begin{onestack}
\fillstack{123}{56}{4}
\stackinput
\end{onestack}
&
\begin{onestack}
\fillstack{123}{564}{}
\stackinput
\end{onestack}
&
\begin{onestack}
\fillstack{1234}{56}{}
\stackoutput
\end{onestack}
\\\hline
\begin{onestack}
\fillstack{1234}{}{56}
\thirdpass
\end{onestack}

&
\begin{onestack}
\fillstack{1234}{5}{6}
\stackinput
\end{onestack}
&
\begin{onestack}
\fillstack{12345}{}{6}
\stackoutput
\end{onestack}
&
\begin{onestack}
\fillstack{123456}{}{}
\stackinoutput
\end{onestack}
\\\hline
\end{tabular}

\caption{Sorting the permutation $356124$ with $k=3$  passes through a stack.}
\label{fig_356124}
\end{center}
\end{figure}

The notion of separated pairs also allows one to more easily study the possible tiers of permutations that require more passes through the stacks. For example, the permutation $\pi=4637251$ contains the four separated pairs $(2, 1), (3, 2), (4, 3),$ and $(6, 5)$ hence $\pi$ is a tier $4$  permutation. Before continuing to investigate separated pairs, we note that the number of separated pairs are preserved in permutation classes.  This in turn shows that the $k$-pass sortable permutations form a permutation class for any value of $k$.

\begin{proposition}~\label{containment}
If $\sigma$ and $\tau$ are two permutations and $\sigma$ is contained in $\tau$ then $\tau$ has at least as many separated pairs as $\sigma$. Equivalently the number of separated pairs in $\sigma$ is less than or equal to the number in any permutation that contains it.
\end{proposition}

\begin{proof}
The proof is simply a generalization of the argument in the proof or Proposition \ref{sep_pair_prop}. Assume $\sigma = \sigma_1 \sigma_2 ...\sigma_k$ is contained in $\tau = \tau_1\tau_2...\tau_n$. Let $\tau_{\sigma} = \tau_{\sigma_1} \tau_{\sigma_2},...\tau_{\sigma_k}$ be the subsequence of $\tau$ corresponding to the pattern $\sigma$. We argue that every separated pair in $\sigma$ forces a distinct separated pair in $\tau$.

Assume $(\sigma_i, \sigma_k)$ is a separated pair, i.e. $\sigma_i = \sigma_k + 1$ and there is a larger entry $\sigma_j$ which separates them. There must be a separated pair in $\tau$ say $(m + 1, m) $ with $\tau_{\sigma_k} \le m < \tau_{\sigma_i}$ by the same argument applied in Proposition~\ref{sep_pair_prop}.

Since $\sigma_i$ and $\sigma_k$ were consecutive integers in $\sigma$, there are no integers in the subsequence $\tau_{\sigma}$ between $\tau_{\sigma_k}$ and $\tau_{\sigma_i}$. Hence for every separated pair in $\sigma$ there is a distinct separated pair in $\tau$.  Thus the tier of $\tau$ is at least that of $\sigma$.
\end{proof}

\begin{corollary}
Given a nonnegative integer $t$, the permutations of tier less than or equal to $t$ form a permutation class.  That is, the $k$-pass sortable permutations form a permutation class for any positive integer $k$.
\end{corollary}

The proof is simply to note that the permutations of tier at most $t$ are those that avoid all permutations of tier $(t+1)$.  This yields another way to see that the $k$-pass sortable permutations form a permutation class for any positive integer $k$.

\subsection{The basis for 2-pass sortable permutations}

We now classify permutations that have maximum tier $t=1$.  That is, if we allow ourselves a maximum of one re-use of the stack to complete the sorting process, which permutations are sortable?

\begin{theorem}
\label{thm_t_1} A permutation $\pi$ is $2$-pass sortable, i.e. $t(\pi) \leq 1$, if and only if $\pi$ avoids
$$ 24153, 24513, 24531, 34251, 35241, 42513,   42531,45231, 261453, 231564, 523164.$$
\end{theorem}
\begin{proof}
First one can check that each of the listed permutations
has exactly two separated pairs.  Moreover, the removal of any entry in any of these permutations also removes at least one separated pair.  Hence by Theorem~\ref{tier_num_pairs}, these permutations are all minimal basis elements for the class of $2$-pass sortable permutations.

Next we note that any permutation contained in the basis for $2$-pass sortable permutations cannot have more than two separated pairs.  For otherwise, if we remove the smallest entry of the smallest separated pair in such a permutation, then the resulting permutation has exactly one less separated pair and is still not $2$-pass sortable.

Hence every basis element has exactly two separated pairs.  For a permutation $\sigma$ to meet the minimal length requirement, one of the separated pairs in $\sigma$ must be $(2,1)$.

In one case, the basis element $\sigma$ has two separated pairs of the form $(2,1)$ and $(3,2)$.  As $3,2,1$ must appear in descending order, the basis elements in this instance are $34251$ and $35241$.

Otherwise, suppose the two separated pairs are $(2,1)$ and $(b+1, b)$ where $b>2$.  There are then six possible relative orders of these elements in a basis element $\sigma$.     For clarity, we consider each of these cases separately.

Case 1:  The permutation $\sigma$ contains the subsequence $(b+1)b21$.  However, then the only other entries $\sigma$ would require are the separators which makes $\sigma=564231$.  However, this permutation is not minimal as it contains $45231$ as a pattern.

Case 2:  The permutation $\sigma$ contains the subsequence $(b+1)2b1$.  Then we need only separate the pair $(b+1,b)$ and thus we have either $45231$ or $42531$.

Case 3:  The permutation $\sigma$ contains the subsequence $(b+1)21b$.  Then we need to separate both pairs, possibly with a single entry between the $2$ and the $1$ if it is large enough.  A larger entry cannot appear immediately following $b+1$ for then $b$ is redundant because $b-1$ also follows this larger entry.  As such, the permutations obtained here are only $42513, 523164$. 

Case 4:  The permutation $\sigma$ contains the subsequence $2(b+1)b1$.  We just need to separate the pair $(b+1,b)$, and the only possibility is $24531$.

Case 5:  The permutation $\sigma$ contains the subsequence $2(b+1)1b$.  Here we also need only separate the pair $(b+1,b)$.  Hence we have $24513$ or $24153$.

Case 6:  The permutation $\sigma$ contains the subsequence $21(b+1)b$.  In this final case, we need to separate both pairs individually.  Thus the basis elements obtained are $231564,251463,261453$.  However, $251463$ is not minimal as it contains $24153$.
\end{proof}

We have included the avoidance numbers for this basis in Column $2$ of Table~\ref{table2} in Section~\ref{gen_fun}, and the data in Table~\ref{table_1} in Section~\ref{bounds} gives the number of permutations of exact tier $t$.

\section{The maximum tier of a permutation of a given length}~\label{bounds}

Continuing the use of separated pairs, we may now show there is a finite basis for each class of $k$-pass sortable permutations by bounding the length of potential basis elements. Given a nonnegative integer $t$, let $P_t$ denote the set of all permutations of tier at most $t$ and let $B_t$ be the basis for this set. That is $\sigma \in P_t$ if and only if $\sigma \in \Av(B_t)$.

\begin{proposition}
Given any $\sigma \in B_t$, the length of $\sigma$ is at most $3(t + 1)$.
\end{proposition}

\begin{proof}
First note that if the tier of $\sigma$ is greater than $t +  1$ then $\sigma$ contains a permutation of tier exactly $t + 1$ which can replace $\sigma$ in the basis. Thus the tier of $\sigma$ is $t + 1$.
If the length of $\sigma \in B_t$ is greater than $3(t + 1)$ then some number occurring in $\sigma$ is not part of a separated pair, nor necessary to separate a pair. Hence there is a shorter permutation which is contained in $\sigma$ of tier $(t  + 1)$ which can replace $\sigma$ in the basis.
\end{proof}

From Theorem~\ref{tier_0}, we know $B_0$ has one length $3$ basis element.  And by Theorem \ref{thm_t_1}, we see the basis $B_1$ consists of eleven basis elements; eight of length $5$ and three of length $6$.   The question of the shortest possible permutation in a given basis is more subtle. Equivalently one can ask, what is the maximal possible tier of a permutation of length $n$?

\begin{notation} Let $\tau(n)$ represent the maximum tier of over all permutations of length $n$, and as before let $t(\sigma)$ represent the  tier of $\sigma$.
\end{notation}

\begin{example}  From Theorems \ref{tier_0} and \ref{thm_t_1}, we have that $\tau(1) = \tau(2) = 0,$ $\tau(3) = \tau(4) = 1$, and $\tau(5) = 2$.  As displayed in Table~\ref{table_1} in Section~\ref{gen_fun}, we have $\tau(6) = 3$, $\tau(7) = \tau(8) = 4$, $\tau(9) = 5$ and $\tau(10) = 6$.
\end{example}

We will prove an exact formula for $\tau(n)$ later in Theorem~\ref{thm_tau},  however we first include proofs of a few lower bounds on $\tau(n)$ to demonstrate some specific constructions.

\begin{lemma}~\label{pack}
For all positive integers $n$, we have $\tau(n) \ge \lfloor \frac{n - 1}{2} \rfloor$.
\end{lemma}

\begin{proof}
One can use $n$ to separate as many distinct pairs as possible.   In particular, we create the separated pairs 
$(2,1),(4,3),(6,5),\ldots,(2\lfloor \frac{n - 1}{2} \rfloor, 2\lfloor \frac{n - 1}{2} \rfloor -1)$.
Such a permutation must have at least $\lfloor \frac{n - 1}{2}\rfloor$ separated pairs and the bound follows.
\end{proof}

\begin{example}  If $n = 7$, use $7$ to separate $(6, 5), (4, 3), (2, 1)$ as in $6427531$.
\end{example}

For odd lengths we can do better.

\begin{lemma} ~\label{odd}
If $n$ is odd with $n = 2k + 1$, then $\tau(n) = k  + \tau(k)$.
\end{lemma}

\begin{proof}
To prove this lemma, we first give a construction for a particular permutation of high tier giving a lower bound for $\tau(n)$, then argue it is optimal.

We create the largest set of consecutive separated pairs we can have, $(k + 1, k), (k, k - 1), (k - 1, k - 2),\ldots,(2, 1)$ and use the numbers from $k + 2$ up to $n = 2k + 1$ to separate them. Since all of these numbers are larger than all of the pairs we are separating, they can be used interchangeably.  In particular, we arrange these separators in an optimal $\tau( \frac{n - 1}{2} ) = \tau(k)$ pattern. Thus the number of separated pairs is at least $k  + \tau(k)$ and hence $\tau(n) \ge  \frac{n - 1}{2}  + \tau( \frac{n - 1}{2} )$.


The construction above produces the maximal number of separated pairs among the first $k + 1$ elements and among the last $k$ entries.  As such if a permutation on $[n]$ has more separated pairs, then it must include the separated pair $(k+2, k+1)$.  However, in order to create a new separated pair $(k + 2, k + 1)$, we need to ``un-separate'' at least one other smaller separated pair as there are not enough larger entries to separate $k+1$ consecutive separated pairs.  At best, we have a net change in tier of zero from the $k+2$ smaller entries.  Further, this cannot increase the number separated pairs we have among the larger entries as they were already in an optimal configuration.

Thus the construction is produces permutations of maximal tier.  Hence $\tau(n) =  \frac{n - 1}{2}  + \tau( \frac{n - 1}{2} )$.
\end{proof}

To illustrate the above we consider a few examples shown below.

\begin{example}
If $n = 7$ we begin with the sequence $4321$ then place $5, 6,$ and $7$ between each pair to separate them. 
We then use the unique optimal $\tau(3)$ pattern (i.e. $231$) for these elements to yield $4637251$ which has optimal tier $4$.
\end{example}

\begin{example}
If $n = 9$, we begin with $54321$ and use $6, 7, 8, 9$ to separate the consecutive pairs in an optimal length $4$ pattern such as $2314$ to yield $574836291$ which has tier $4 + 1 = 5$ and $\tau(9) = 5$.
\end{example}

We also note that increasing the allowed length of a permutation by one increases the maximal tier by at most one.

\begin{lemma}~\label{at_most_1}
For any positive integer $n$ we have $\tau(n + 1) \le \tau(n) + 1$.
\end{lemma}

\begin{proof}
Let $\sigma$ be in $S_{n + 1}$ with tier $\tau(n + 1)$. If we remove the number $1$ from $\sigma$ and reduce all of the remaining numbers by $1$ to create a permutation $\rho \in S_n$. By removing $1$ from $\sigma$ we have at most removed one separated pair, hence $t(\rho) \ge t(\sigma) - 1$ or $\tau(n + 1) = t(\sigma) \le t(\tau) + 1 \le \tau(n) + 1.$
\end{proof}

Also, since we can find a length $n + 1$ permutation containing a given length $n$ permutation, we can use the argument in Lemma~\ref{odd} to get a lower bound for the case when $n$ is even as well.

\begin{corollary}
If $n = 2k$ then $\tau(n) \ge k - 1 + \tau(k - 1)$.
\end{corollary}

However, while we still can have only $\frac{n }{2}-1=\lfloor \frac{n - 1}{2} \rfloor$ separated pairs from our smallest $\frac{n}{2}+1$ entries, we can duplicate the above construction and have $\tau(\frac{n}{2})$ separated pairs amongst our largest $\frac{n}{2}$ entries.  Hence we have the following lemma.

\begin{lemma}~\label{even}
If $n = 2k$ then $\tau(n) \ge k - 1 + \tau(k)$.
\end{lemma}

Combining Lemmas~\ref{odd}, \ref{at_most_1}, and \ref{even}, we get the following result.

\begin{theorem}\label{thm_tau}
The maximum tier of a permutation of length $n$ satisfies the recurrence 
\begin{equation}\label{tau_0}
\tau(n) = \left\lfloor \frac{n - 1}{2} \right\rfloor + \tau\left(\left\lfloor \frac{n}{2} \right\rfloor\right).
\end{equation}
Moreover
\begin{equation}\label{tau}
\tau(n) = \sum_{j \geq 1} \bigg{\lfloor} \frac{n - 2^{j - 1}}{2^j} \bigg{\rfloor} = n - 1 - \lfloor \log_2(n) \rfloor.
\end{equation}
\end{theorem}

\begin{proof}
To prove this, we first show that the right hand sides of (\ref{tau_0}) and (\ref{tau}) give lower bounds for $\tau(n)$. We then argue the bound is exact for $n$ a power of $2$ and combined with Lemma \ref{at_most_1} the equality follows.

The recursive bound for $\tau(n)$, 
\[ 
\tau(n) \ge \left\lfloor \frac{n - 1}{2} \right\rfloor + \tau\left(\left\lfloor \frac{n}{2} \right\rfloor\right)
\]
 simply combines Lemmas \ref{odd}, \ref{even}. 

The right hand side  of (\ref{tau}) as a lower bound follows from iterating that recursive formula. For example,
\[
\displaystyle{
\tau(n) \ge \left\lfloor \frac{n - 1}{2} \right\rfloor + \left\lfloor \frac{\lfloor \tfrac{n}{2}\rfloor - 1}{2} \right\rfloor + \tau\left(\left\lfloor \frac{\lfloor \tfrac{n}{2}\rfloor}{2} \right\rfloor\right).}
\]
Let $\displaystyle{n = \sum_{i = 0}^{k} c_i 2^i}$ be the binary expansion of $n$ where $c_k = 1$ and note that in general
\[
\left\lfloor \frac{\lfloor \tfrac{n}{2^i}\rfloor}{2} \right\rfloor = \left\lfloor \frac{c_i+ c_{i + 1}2 + c_{i + 2}2^2+\cdots+c_{k}2^{k-i}}{2} \right\rfloor = c_{i + 1} + c_{i + 2}2 + c_{i+3}2^2 +\cdots+c_{k}2^{k-i-1} = \left\lfloor \frac{n}{2^{i + 1}} \right\rfloor
\]
and we have
\[
\left\lfloor \frac{\lfloor \tfrac{n}{2^i}\rfloor - 1}{2} \right\rfloor = \left\lfloor \frac{n - 2^i}{2^{i + 1}} \right\rfloor
{\rm ~ since~} \left\lfloor \frac{n}{2^i} \right\rfloor - 1 = \left\lfloor \frac{n - 2^i}{2^i} \right\rfloor
\]
so the right hand side of (\ref{tau}) follows as a lower bound for $\tau(n)$.

To show  (\ref{tau}),  we evaluate the individual terms in the sum
\begin{align*}
\left\lfloor \frac{n - 2^{j - 1}}{2^j} \right\rfloor &= \left\lfloor \frac{c_0}{2^j} + \frac{c_1}{2^{j - 1}} + \cdots +\frac{c_{j - 1}}{2} + c_j + c_{j+1}2 + \cdots c_k2^{k - j} -\frac{1}{2} \right\rfloor \\
&= c_j + c_{j+1}2 +\cdots+c_k2^{k - j} + \left\lfloor \frac{c_0}{2^j} + \frac{c_1}{2^{j - 1}} + \cdots +\frac{c_{j - 1}}{2} -\frac{1}{2} \right\rfloor.
\end{align*}
Finally we note that
\begin{align*}
\left\lfloor \frac{c_0}{2^j} + \frac{c_1}{2^{j - 1}} + \cdots + \frac{c_{j - 1}}{2} -\frac{1}{2} \right\rfloor
 &= \left\{ \begin{array}{cc} 0 & c_{j - 1} = 1\\ -1 & c_{j - 1} = 0 \end{array}\right. \\
 &= c_{j - 1} - 1.
\end{align*}
Thus $\displaystyle{\bigg{\lfloor} \frac{n - 2^{j - 1}}{2^j} \bigg{\rfloor} =  \left[\sum_{i = j}^k c_i 2^{i - j}\right] + (c_{j - 1} - 1)}$.
Summing on $j$ yields
\begin{align*}
\tau(n) & \ge  \displaystyle{ \sum_{j = 1}^k \bigg{\lfloor} \frac{n - 2^{j - 1}}{2^j} \bigg{\rfloor} 
= \sum_{j = 1}^k \bigg{[} \bigg{(} \sum_{i=j}^k c_i2^{i-j} \bigg{)} + c_{j-1} - 1 \bigg{]} } 
= \displaystyle{ \sum_{j = 1}^k  \sum_{i=j}^k c_i2^{i-j}  + \sum_{j = 1}^k (c_{j-1} - 1)  } \\
&= \displaystyle{ \sum_{i = 1}^k  \sum_{j=1}^i c_i2^{i-j}  + \sum_{j = 0}^{k-1} c_{j} - k}  
= \displaystyle{ \sum_{i = 1}^k c_i \sum_{j=0}^{i-1} 2^{j}  + \sum_{j = 0}^{k-1} c_{j} - k }  
= \displaystyle{ \sum_{i = 1}^k c_i (2^i -1)  + \sum_{j = 0}^{k-1} c_{j} -  k }  \\
&= \displaystyle{ \sum_{j = 0}^k c_j (2^j -1)  + \sum_{j = 0}^{k} c_{j} -c_k - k }  
= \displaystyle{ \sum_{j = 0}^k c_j 2^j   -c_k - k  } 
= n   -1 - k   \\
&=n - 1 - \lfloor \log_2(n) \rfloor.
\end{align*}

To complete the argument, we show first if $n$ is a power of $2$, the right hand side of (\ref{tau}) is also an upper bound for $\tau(n)$.  We then combine this with the lower bound and Lemma \ref{at_most_1} and the equalities follow. 

First we note that $\tau(2) = 0$ by inspection. Now assume $n = 2^k$ for a positive integer $k$ and that 
$\tau(n) > (n-1) - \lfloor \log_2(n) \rfloor = (2^k - 1) - k$ and $\pi \in S_n$ has $t(\pi) = \tau(n)$. 

As seen in previous constructions, among the first $2^{k - 1}$ entries of $\pi$ there are at most $2^{k - 1} - 1$ elements of separated pairs. (If $1, 2, 3, ...2^{k - 1} - 1$ are such that each is the smaller entry in a separated pair and $2^{k-1}$ appears before $2^{k-1} -1$, then $2^{k - 1}$ must appear as either the first or second element of $\pi$ and cannot be the smaller entry of a separated pair.)  Hence the removal of the entries $1, 2, 3, ... 2^{k - 1}$ from $\pi$ removes at most $2^{k - 1} - 1$ separated pairs. 

By then reducing each of the remaining entries of $\pi$ by $2^{k - 1}$ (but retaining the order), we obtain a permutation of length $2^{k - 1}$ with tier greater than $2^k - 1 - k - (2^{k - 1} - 1) = 2^{k - 1} - 1 - (k - 1) = \frac{n}{2} - 1 - \left\lfloor \log_2\left(\frac{n}{2}\right)\right\rfloor$. Thus if $\tau(2^k) >  (2^k - 1) - k$, then $\tau(2^{k-1}) >  (2^{k-1} - 1) - (k-1)$, and so inductively there must be a permutation of length $2$ and tier $1$ which is a contradiction.

Now if $n$ is not a power of $2$, then let $k$ be the greatest integer such that $n > 2^k$.  Suppose $n = 2^k +m$.  Then
\begin{align*}
\tau(n) &\leq \tau(2^k) +m \quad & \text{by Lemma \ref{at_most_1}}\\
&= (2^k - 1 -k) +m \quad &  \text{from the result we just proved}\\
&= (2^k+m) - 1 -\lfloor \log_2(n) \rfloor\\
&= (n - 1) -\lfloor \log_2(n) \rfloor
\end{align*}

Therefore $\tau(n) =  n - 1 - \lfloor \log_2(n) \rfloor$ for all $n$.  Moreover, this equality also forces the first bound, $\tau(n) = \left\lfloor \frac{n - 1}{2} \right\rfloor + \tau\left(\left\lfloor \frac{n}{2} \right\rfloor\right)$ to be sharp which completes the proof.
\end{proof}

From Theorem \ref{thm_tau} we get a curious result for lengths of the form $n = 2^k - 1$.

\begin{corollary}
Let $n = 2^k - 1$ for some integer $k$, then every permutation of tier $\tau(n)$ is constructed with the method given in Lemma \ref{odd}. Moreover there is exactly one such permutation.
\end{corollary}

\begin{proof} Given the formula in Theorem \ref{thm_tau}, if $n = 2^k -1$ then $\tau(n + 1) = \tau(n)$. Assume  that $\sigma$ is a permutation of tier $\tau(n)$ and note $\tfrac{n - 1}{2} = 2^{k - 1} - 1$ which we will call $m$ for convenience. Now suppose $\sigma$ is not of the form $(m + 1)a_2(m) a_4(m - 1) a_6 \ldots (2)a_{n - 1}(1)$ where $a_2a_4 \ldots a_{n - 1}$ has tier $\tau(2^{k - 1} - 1)$. Let $j$ is the first element that is not part of a separated pair $(j + 1, j)$. 

To avoid the form described above, we must have $j < m + 1$.  We can create a new permutation $\hat{\sigma}$ by first moving all entries less than $j + 1$ to the right (if necessary) so that there is only one separator for each separated pair, then increase all values in $\sigma$ larger than $j$ by one, and add a $j + 1$ in the first position. We have that $(j + 1, j)$ is now a separated pair in $\hat{\sigma}$ without affecting any other separated pairs, thus $t(\hat{\sigma}) = t(\sigma) + 1 = \tau(n) + 1$ contradicting the result in Theorem \ref{thm_tau}. 

Thus every maximal tier permutation of length $2^k - 1$ has the form above.  Thus the number of such permutations is the same as the number of length $2^{k - 1} - 1$ permutations $a_2a_4 \ldots a_{n - 1}$ having maximal tier.  Proceeding inductively, the number of such permutations is seen to be the same as the number of length one permutations of maximal tier and the result follows.
\end{proof}

\begin{example} For length $n = 2^4 - 1 = 15$ the maximal tier is $11$ and the unique permutation of this length and tier is $\pi = 8\; 12\; 7\; 14\; 6\; 11\; 5\; 15\; 4\; 10\; 3\; 13\; 2\; 9\; 1$.
\end{example}

We can also consider how to find the tier of a permutation obtained by combining two permutations via a specific kind of concatenation to get a new permutation.

\begin{definition}  An \emph{interval} of a permutation $\pi$ is a consecutive subsequence of $\pi$ that contains consecutive values.
\end{definition}

\begin{example}  The permutation $\sigma = 685712943$ contains maximal intervals $6857, 12, 9, 43$.
\end{example}

\begin{definition}
A permutation $\pi$ is said to be \emph{plus-decomposable} if $\pi$ is the concatenation of two non-empty intervals $\omega$ and $\tau'$ where the values of $\omega$ are less than those of $\tau'$.  Further, if we rescale the entries of $\tau'$ by subtracting the length of $\omega$ from each entry of $\tau'$ to get a permutation $\tau$, we denote $\pi = \omega \oplus \tau$.  If a permutation is not plus-decomposable, we say the permutation is \emph{plus-indecomposable}.
\end{definition}

\begin{example}  The permutation $\pi=43126758$ is plus-decomposable and can be written as $\pi = 4312 \oplus 231\oplus 1$.  The permutation $\sigma = 685712943$ is plus-indecomposable.
\end{example}

\begin{proposition} If a permutation $\pi$ is plus-decomposable, say $\pi = \sigma \oplus \tau$, then the tier of $\pi$ is the sum of the tiers of $\sigma$ and $\tau$, i.e. $t(\pi)= t(\sigma) + t(\tau)$.
\end{proposition}

\begin{proof}  Consider the process of sorting $\pi$.  As every entry of the $\sigma$ portion of $\pi$ must be pushed to the output before any entry of $\tau'$, we cannot start sorting $\tau'$ until the last pass needed to sort $\sigma$ has commenced.  This final pass where part of $\sigma$ is still being sorted is the $(t(\sigma)+1)$st pass.

During the $(t(\sigma) +1)$st pass, every remaining entry of the $\sigma$ portion can be output before $\tau'$  is considered and thus this pass can also be used as the first pass for $\tau'$.  Then $\tau'$ requires $t(\tau)$ more passes to be sorted.  Hence $\pi$ is $[t(\sigma) + t(\tau) +1]-$pass sortable and thus $t(\pi) = t(\sigma) + t(\tau)$.
\end{proof}

\begin{definition}
A permutation $\pi$ is said to be \emph{minus-decomposable} if $\pi$ is the concatenation of two non-empty intervals $\omega'$ and $\tau$ where the values of $\omega'$ are greater than those of $\tau$.  As before, if we rescale the entries of $\omega'$ by subtracting the length of $\tau$ from each entry of $\omega'$ to get a permutation $\omega$, we denote $\pi = \omega \ominus \tau$.
\end{definition}

\begin{example}  The permutation $\pi=67584312$ is minus-decomposable and can be written as $\pi = 2314 \ominus 1\ominus 1 \ominus 12$.
\end{example}

\begin{proposition} If a permutation $\pi$ is minus-decomposable, say $\pi = \sigma \ominus \tau$, and $\sigma$ is $k$-pass sortable and $\tau$ is $m$-pass sortable, then $\pi$ is $(k+m)$-pass sortable, i.e. $t(\pi)= t(\sigma) + t(\tau) + 1$ if the last entry of $\sigma$ is not $1$, otherwise $\pi$ is $(k + m - 1)$ sortable and $t(\pi) = t(\sigma) + t(\tau)$.
\end{proposition}

\begin{proof}  Again, consider the process of sorting $\pi$.  Every entry of $\tau$  must be pushed to the output before any entry of $\sigma'$, we cannot start sorting $\sigma'$ until after the last pass needed to sort $\tau$ has been completed.  If the last entry of $\sigma'$ is the smallest entry of $\sigma'$ then the last pass to sort $\tau$ is also the first pass to sort $\sigma'$, hence $\pi$ is $(k + m)-$pass sortable and $t(\pi) = t(\sigma) + t(\tau)$. Otherwise another pass is required to begin sorting $\sigma'$ and  hence $\pi$ is $[k + m + 1]-$pass sortable and thus $t(\pi) = t(\sigma) + t(\tau) + 1$.
\end{proof}

Note that if a permutation is minus-indecomposable and length greater than one, then it cannot have its smallest value in the last position. This leads to the following corollary:

\begin{corollary}
Let $\pi$ be a permutation and assume $\pi = \sigma_1 \ominus \sigma_2 \ominus \cdots \ominus \sigma_p$ where each $\sigma_i$ is minus-indecomposable, and let $r$ denote the number of the $\sigma_i$ with $i < p$ of length one, then $t(\pi) = p - r - 1 + \displaystyle{\sum_{i = 1}^p t(\sigma_i)}$.
\end{corollary}

\section{Exact enumeration of tier t permutations of length n}~\label{gen_fun}

A simple program was written in SAGE \cite{sagemath} to compute the tier of all permutations up to length $10$. The data for the number of permutations of a given length and exact tier is given in Table \ref{table_1}. The numbers given in this triangular form appear in the OEIS A122890 and A158830 \cite{OEIS}. The sequences found in the OEIS were created by manipulating generating series for iterated functions, however one version (A122890) does have an equivalent  combinatorial interpretation which can be found in Parker's thesis~\cite{Parker_thesis}. We modify Parker's description slightly to align the data properly (in particular effecting a row-reversal which gives OEIS A158830).

\begin{table}
\[
\begin{tabular}{|r||c|c|c|c|c|c|c|}
\hline
       & t = 0 	& t = 1 	& t = 2		& t = 3 		& t = 4 		& t = 5		& t = 6 	 		 \\ \hline
n = 1  &	1	  &			&			&				&				&			 &				\\ \hline
n = 2  &	2 	  & 		&			&				&				&			&				 \\ \hline
n = 3  &	5	  &  1		&			&				&				&			 &				 \\ \hline
n = 4  &	14	  & 10		&			&				&				&			 &				  \\ \hline
n = 5  &	42	  & 70		&	8		&				&				&			 &				  \\ \hline
n = 6  & 	132	  & 424		&	160		& 4				&				&			 &				   \\ \hline
n = 7  &	429	  & 2382	&	1978	& 250			& 1				&			&				   \\ \hline
n = 8  &	1430  &	12804	& 19508		& 6276			& 302			&			&				  \\ \hline
n = 9  &	4862  &	66946	& 168608	& 106492		& 15674			&	298		&			   \\ \hline
n = 10 &	16796 &	343772	& 1337684	& 1445208		& 451948		&	33148	&	244			 \\ \hline
\end{tabular}
\]
\caption{Number of permutations of length $n$ and exact tier $t$}
\label{table_1}
\end{table}

\subsection{Parker's original description of OEIS A122890 and a bijection}

Assume $n, t$  are positive integers and let $W(n, t)$ be the number of sequences $a_1a_2\ldots a_n$ of length $n$,  such that each $1\le a_i \le i$ and there are exactly $t$ indices $i$ such that $a_i \le a_{i+1}$.

\begin{example}~\label{W42}  We have $W(4,2)=10$ as it counts the sequences:
$$1121,1131,1132,1211,1212,1213,1214,1221,1231,1232.$$
\end{example}

For convenience, we reindex the sequence and consider the complementary condition on the indices. That is, assume $n, t$ are non-negative integers with $n > 0$, and let $T(n,t)$ be the number of sequences $a=a_na_{n - 1}\ldots a_1$ of length $n$ where $1\le a_{n - i + 1} \le i$ for each $i$, and where there are exactly $t$ values of $i$ such that $a_{i + 1} > a_{i}$.  Note the reversal in indexing the entries of the sequence $a$.

We say the sequence has a \emph{descent} at $i$ if $a_{i + 1} > a_{i}$.  In this language $T(n, t)$ is the number of length $n$ sequences with entries bounded by the index (as described above) such that there are exactly $t$ descents. That is, $T(n,t) = W(n, n - 1 - t)$.  Note that the reindexing has no effect on the problem and counting descents instead of non-descents reverses the rows of the data in OEIS A122890.  We will refer to these reindexed sequences as \emph{Parker sequences}.

\begin{example} We have $T(4,1) = W(4,2)=10$ and counts the sequences given in Example~\ref{W42}.
\end{example}

\begin{example} $T(5,2)=8$ and counts the sequences:
$$11321,12121,12131,12132,12141, 12142,12143, 12321.$$
\end{example}

In the proof of Theorem~\ref{enumeration} we will create a bijection between permutations of length $n$ and tier $t$ and the Parker sequences counted by $T(n, t)$ where separated pairs in permutations correspond to descents in Parker sequences.  In particular, a permutation having $(i + 1, i)$ as a separated pair will correspond to a Parker sequence with a descent at index $i + 1$ where the indexing is from right to left as described above.  

\begin{example}
 In the proof of Theorem~\ref{enumeration}, the permutation $\pi = 53412678$ having separated pair $(3,2)$ will be shown to correspond to the Parker sequence $12344545$ which has a descent at index $3$ (from the right).
\end{example}

\begin{theorem}
\label{enumeration}
The number of permutations of length $n$ and tier $t$ is $T(n, t)$.
\end{theorem}

\begin{proof}
Let $P = \{ a_na_{n - 1}\ldots a_1 ~ | ~1\le a_{n - i + 1}\le i, \forall i, n\}$, that is $P$ is the set of all Parker sequences, and let $S$ be the set of all permutations. Let $f: P\rightarrow S$ be defined as $f(a_{n}a_{n - 1}\ldots a_1) = \pi_n\pi_{n - 1}\ldots \pi_1$ where the element $1$ is placed in $\pi_{a_1}$.  Then for each $i$ beginning with $2$ and proceeding in order to $n$, place the element $i$ in the $a_{i}$th position  of the remaining positions in $\pi$, counting from the right.

The map is a bijection as the process is invertible. That is, given a permutation $\pi = \pi_n\pi_{n - 1}\ldots \pi_1$, create the sequence $f^{-1}(\pi)=a_na_{n - 1}\ldots a_1$ by letting $a_i$ be the relative position (from right to left) of the element $i$ among the elements greater or equal to $i$ in the permutation $\pi$. The bounds on the entries must obey Parker's restriction as there are $n - i + 1$ numbers greater or equal to $i$ in $\pi$.

We now prove that there is a descent at $i$ for the sequence $a$ if and only if $(i + 1, i)$ is a separated pair in $f(a)$. Assume there is a descent at index $i+1$, that is $a_{i + 1} > a_{i}$ since $a_{i+1}$ appears before $a_i$ in $a$.  This would imply that $i  + 1$ is placed to the left of $i$ in the permutation $f(a)$.  Since the inequality is strict there must also be at least one unoccupied position between these two elements after we place $i + 1$ in the permutation. The only elements remaining are larger than $i + 1$.  Hence at least one of these larger elements must separate $i + 1$ and $i$ in $f(a)$.

Conversely, assume there is a separated pair in $\pi$, say $(i + 1, i)$.  Consider the associated sequence $a=f^{-1}(\pi)$. Then $a_{i + 1} \ge a_i$ since every element to the right of $i$ that is larger than $i$ is also to the right of $i+1$ and larger than $i+1$.   Further, since there must be a larger element separating $i + 1$ and $i$ in $\pi$, we have $a_{i + 1} >a_{i}$.  Hence there is a descent in $a$ at index $i+1$.

Since the number of descents a Parker sequence has is the same as the number of separated pairs the associated permutation has, we have the number of permutations of length $n$ and tier $t$ is $T(n, t)$.
\end{proof}

\begin{example}  As an example, we compute $f(12133)$ for the bijection given in the proof of Theorem~\ref{enumeration}.  
\begin{enumerate}
\item Begin with the last number in the sequence $3$ and place the $1$ into the third position from the right in the permutation so we have $**1**$. 
\item Then we consider the next element in the sequence from the right, since it is also $3$, we place the element $2$ in the third remaining position from the right and we have $*21**$. 
\item The next element in the sequence is $1$ hence we place the element $3$ in the first remaining position from the right, i.e. $*21*3$. 
\item Then we place the element $4$ in the second remaining position from the right as the next entry is $2$ to get $421*3$.  
\item Finally $5$ must be placed in the only remaining position (which is the first from the right) to get $42153$.
\end{enumerate}
\end{example}

\begin{corollary}
\label{enumeration_class}
The number of permutations of length $n$ sortable by a stack with (at most) $k$ passes, or $Av(B_{k - 1})$,  is $\displaystyle{\sum_{j=0}^{k - 1}T(n,j)}$.
\end{corollary}

\subsection{An explicit formula for the number of tier t permutations of length n}

The sequences considered in the OEIS were originally constructed by manipulating generating functions. For the sake of completeness we include the constructions here. We also add a new construction of the generating function that gives an explicit formula for $T(n,t)$.

{\bf OEIS 158830 Construction}

Let $\hat{C}(x) = x C(x)$ where $C(x)$ is the generating function for the Catalan numbers or equivalently the avoidance numbers of any permutation of length $3$. Let the numbers in the $n$th row be the coefficients of the $n$th iterate of $\hat{C}(x)$. The multiply the generating function represented by the $t$th column by $(1 - x)^t$. The resulting entry in row $n$ and column $t$ is $T(n, t)$.

{\bf OEIS 122890 Construction, Parker \cite{Parker_thesis}}

Let $a_1(x) = x$, $a_2(x) = x + x^2$ and for all $n \ge 2$, let $a_n(x)$ be the $n - 1$ iterate of $x + x^2$. Write the coefficients of $a_n(x)$ as the entries in row $n$ and multiply the generating function represented by column $j$ by $(1 - x)^j.$ $T(n, t)$ is the entry in position $(n, n - t)$.

{\bf Alternate Construction}

To develop a recurrence for the number of permutations of length $n$ and tier $t$, we introduce a set of functions $f_k$ from permutations of length $n$ to length $n + 1$. In particular, given a permutation $\alpha$ of length $n$, let $f_k(\alpha)$ be the permutation obtained by increasing every number in $\alpha$ by one, then inserting a one into the $k$th position. For example $f_2(3412) = 41523$ and $f_3(3412) = 45123$. 

Assume $\beta = f_k (\sigma)$ for some permutation $\sigma$  of length $n$ and some integer $0 \le k \le n$. We note $\beta$ has the same number of separated pairs $(i + 1, i)$ with $i\ge 2$ as $\sigma$. Thus the $t(\beta)$ is either $t(\sigma)$ or $t(\sigma) + 1$ depending  on whether $(2, 1)$ is a separated pair in $\beta$. We note that $(2, 1)$ is a separated pair in $\beta$ if and only if $k$ is at least two larger than the position of the $1$ in $\sigma$.  That is, if the $1$ in $\sigma$ occurs in the $j$th position and $k \ge j + 2$, we will have $(2,1)$ as a separated pair in $\beta = f_k(\sigma)$.

Let $P(n, t, k)$ be the number of permutations of length $n$, tier $t$ (where $n > 0, t \ge 0$), and have the $1$ in the $k$th position (thus $1\le k \le n$). Clearly we have $T(n, t) = \displaystyle{\sum_{k = 1}^n P(n, t, k)}$.

\begin{example}
For example for length $n = 3$ we have $P(3, 0, 1) = 2$ (for the permutations $123, 132$), $P(3, 0, 2) = 2$ (for $213, 312$), $P(3, 0, 3) = 1$ (for $321$), $P(3, 1, 1) = 1$ (for $231$), and otherwise $P(3, t, k) = 0$.
\end{example}

\begin{theorem} \label{recurrence}
For all integers $n > 0$, $t \ge 0$, $1\le k \le n  + 1$ we have
\[
P(n + 1, t, k) = \sum_{j \ge k - 1} P(n, t, j) + \sum_{ j \le k - 2} P(n, t - 1, j).
\]
\end{theorem}
\begin{proof}
First note that every permutation $\beta$ of length $n + 1$ with a $1$ in the $k$-th position arises exactly once from applying one of the $f_k$ operators to a permutation $\sigma$ of length $n$.  Consider $\beta = f_k(\sigma)$. If we apply $f_k$ to a permutation then the tier is either fixed or increases by one. The tier is fixed exactly when $(2, 1)$ is not a new separated pair in the permutation of length $n + 1$.  That is, the tier increases when the $1$ of $\beta$ is appears at least two slots to the right of the $1$ of $\sigma$ so that there is larger element to separate the pair $(2,1)$ of $\beta$.
\end{proof}

\begin{example}
Consider the permutation $24153$, since the $1$ occurs in the third position, $$t(f_k(24153)) = \left\{ \begin{array}{cc} t(24153) = 2 & {\rm if~ }k \le 4 \\ t(24153) + 1 = 3 & {\rm if~} k\ge 5 \end{array}\right.$$
\end{example}

\begin{table}
\[
\begin{tabular}{|r||c|c|c|c|c|c|c|}
\hline
       & t = 0 	& t $\le$ 1 	& t $\le$ 2		& t $\le$ 3 		& t $\le$ 4 		& t $\le$ 5		 & t $\le$ 6 			\\ \hline
n = 1  &	1	  &	1		&	1		&	1			&	1			&	1		 &	1		 \\ \hline
n = 2  &	2 	  & 2		&	2		&	2			&	2			&	2		&	 2		\\ \hline
n = 3  &	5	  &  6		&	6		&	6			&	6			&	6		 &	6		 \\ \hline
n = 4  &	14	  & 24		&	24		&	24			&	24			&	24		 &	24			  \\ \hline
n = 5  &	42	  & 112		&	120		&	120			&	120			&	120		 &	120			  \\ \hline
n = 6  & 	132	  & 556		&	716		& 720			&	720			&	720		&	 720		  \\ \hline
n = 7  &	429	  & 2811	&	4789	& 5039			& 5040			&	5040	&	5040	   \\ \hline
n = 8  &	1430  &	14234	& 33742		& 40018			& 40320			&	40320	& 40320		  \\ \hline
n = 9  &	4862  &	71808	& 240416	& 346908		& 362582		&	362880	&	362880		 \\ \hline
n = 10 &	16796 &	360568	& 1698252	& 3143460		& 3595408		&	3628556	& 3628800			 \\ \hline
\end{tabular}
\]
\caption{Number of permutations of length $n$ and tier at most $t$}
\label{table2}
\end{table}

The data in Table \ref{table2} gives the number of elements in each permutation class at a given length and up to a given tier. We also note that we were able to compute the number of basis elements in $B_3$. There are $4$ of length 6, $116$ of length $7$, $67$ of length $8$ and $12$ of length $9$ (note the maximal length of a basis element would be $12$).

Now, we are ready to find an explicit formula for the generating function $T_t(x)=\sum_{n\geq0}T(n,t)x^n$. In order to do that, we define
$P_{n,t}(v)=\sum_{k=1}^nP(n,t,k)v^{k-1}$. By multiplying the
recurrence relation in the statement of Theorem \ref{recurrence} by $v^{k-1}$,
we have
$$\sum_{k=1}^{n+1}P(n+1,t,k)v^{k-1}=\sum_{k=1}^{n+1}v^{k-1}\sum_{j\geq
k-1}P(n,t,j)+\sum_{k=1}^{n+1}v^{k-1}\sum_{j\leq k-2}P(n,t-1,j),$$
which, by exchanging the order of the sums, implies
$$\sum_{k=1}^{n+1}P(n+1,t,k)v^{k-1}=P_{n,t}(1)+\sum_{k=1}^n\frac{v(1-v^k)}{1-v}P(n,t,k)
+\sum_{k=1}^n\frac{v^{k+1}-v^{n+1}}{1-v}P(n,t-1,k).$$ Thus, by
definitions of $P_{n,t}(v)$, we obtain
\begin{align}
P_{n+1,t}(v)=P_{n,t}(1)+\frac{v}{1-v}(P_{n,t}(1)-vP_{n,t}(v))+\frac{v^2}{1-v}(P_{n,t-1}(v)-v^{n-1}P_{n,t-1}(1)).\label{eqtt1}
\end{align}
Define $P_n(w,v)=\sum_{t\geq0}P_{n,t}(v)w^t$. Then, by multiplying
\eqref{eqtt1} by $w^t$ and summing over $t\geq1$, we obtain
\begin{align}
P_{n+1}(w,v)=P_{n}(w,1)+\frac{v}{1-v}(P_{n}(w,1)-vP_{n}(w,v))+\frac{v^2w}{1-v}(P_{n}(w,v)-v^{n-1}P_{n}(w,1))\label{eqtt2}
\end{align}
with $P_1(v,w)=1$. Now, we define
$P(z;w,v)=\sum_{n\geq1}P_n(w,v)z^n$ to be the generating function
for $P_n(w,v)$. By multiplying \eqref{eqtt2} by $z^n$ and summing
over $n\geq1$, we have
$$P(z;w,v)-z=zP(z;w,1)+\frac{vz}{1-v}(P(z;w,1)-vP(z;w,v))+\frac{vwz}{1-v}(vP(z;w,v)-P(vz;w,1)),$$
which is equivalent to
$$\left(1+\frac{v^2z(1-w)}{1-v}\right)P(z;w,v)=z+\frac{z}{1-v}P(z;w,1)-\frac{vwz}{1-v}P(vz;w,1).$$
To solve this functional equation, we apply the kernel method and
take $$v=C(z(1-w))=\frac{1-\sqrt{1-4z(1-w)}}{2z(1-w)},$$ which
cancels $P(z;w,v)$, where $C(z)=\frac{1-\sqrt{1-4z}}{2z}$ is the
generating function for the Catalan numbers
$\frac{1}{n+1}\binom{2n}{n}$. This gives
\begin{align}
P(z;w,1)=C(z(1-w))-1+wC(z(1-w))P(zC(z(1-w));w,1).\label{eqtt3}
\end{align}

Define $\rho_0(z) = C(z(1 - w))$ and $\rho_j(z)=C\left(z(1-w)\prod_{i=0}^{j-1}\rho_i(z)\right)$
for all $j\geq 1$. Then, by assuming $0<|w|<1$, $|z|<1$ and
iterating \eqref{eqtt3}, we have
\begin{align*}
P(z;w,1)
&=\rho_0(z)-1+w\rho_0(z)P(z\rho_0(z);w,1)\\
&=\rho_0(z)-1+w\rho_0(z)(\rho_0(z\rho_0(z))-1)+w^2\rho_0(z)\rho_0(z\rho_0(z))P(z\rho_0(z)\rho_0(z\rho_0(z));w,1)\\
&=\rho_0(z)-1+w\rho_0(z)(\rho_1(z)-1)+w^2\rho_0(z)\rho_1(z)P(z\rho_0(z)\rho_1(z);w,1)\\
&=\cdots,
\end{align*}
which leads to the following result.

\begin{theorem}\label{thm}
The generating function $$P(z;w,1)=\sum_{n\geq1}\sum_{t\geq0}\sum_{k=1}^nP(n,t,k)w^tz^k=\sum_{n\geq1}\sum_{t\geq0}T(n,t)w^tz^n$$
is given by
$$T(z,w)=\sum_{j\geq0}(\rho_j(z)-1)w^j\prod_{i=0}^{j-1}\rho_i(z).$$
\end{theorem}

Define $\psi_j(z)=\sqrt{2\psi_{j-1}(z)-1}$ with
$\psi_1(z)=\sqrt{1-4z(1-w)}$ and $\psi_0=1-2z(1-w)$. By induction
on $j\geq0$, we obtain
$$\rho_j(z)=\frac{1-\psi_{j+1}(z)}{1-\psi_j(z)}.$$
Thus, $\prod_{i=0}^{j-1}\rho_i(z)=\frac{1-\psi_{j}(z)}{2z(1-w)}$
for all $j\geq0$. Hence, by Theorem \ref{thm}, we have the
following formula.

\begin{theorem}\label{thm1}
The generating function $P(z;w,1)=T(z,w)$ is given by
$$T(z,w)=\sum_{j\geq0}\frac{\psi_j(z)-\psi_{j+1}(z)}{2z(1-w)}w^j,$$
where  $\psi_j(z)=\sqrt{2\psi_{j-1}(z)-1}$ with
$\psi_1(z)=\sqrt{1-4z(1-w)}$ and $\psi_0=1-2z(1-w)$.
\end{theorem}

In order to find the generating function for
$T_t(z)=\sum_{n\geq1}T(n,t)z^n$, we have to find the coefficient
of $w^t$ in $T(z,w)$. Thus, by Theorem \ref{thm1}, we have
\begin{align}
T_t(z)=\sum_{j=0}^t[w^{t-j}]\left(\frac{\psi_j(z)-\psi_{j+1}(z)}{2z(1-w)}\right).\label{eqtt4}
\end{align}
For example, for $t=0$, we have
$$T_0(z)=[w^0]\left(\frac{\psi_0(z)-\psi_1(z)}{2z(1-w)}\right)=\frac{1-2z-\sqrt{1-4z}}{2z}=C(z)-1,$$
as expected. For $t=1$, we have
\begin{align*}
T_1(z)&=[w^1]\left(\frac{\psi_0(z)-\psi_1(z)}{2z(1-w)}\right)+[w^0]\left(\frac{\psi_1(z)-\psi_2(z)}{2z(1-w)}\right)\\
&=[w^1]\left(\frac{1-\sqrt{1-4z(1-w)}}{2z(1-w)}-1\right)+[w^0]\left(\frac{\sqrt{1-4z(1-w)}-\sqrt{2\sqrt{1-4z(1-w)}-1}}{2z(1-w)}\right)\\
&=-\frac{1}{\sqrt{1-4z}}+\frac{1-\sqrt{1-4z}}{2z}+\left(\frac{\sqrt{1-4z}-\sqrt{2\sqrt{1-4z}-1}}{2z}\right)\\
&=\frac{1-\sqrt{2\sqrt{1-4z}-1}}{2z}-\frac{1}{\sqrt{1-4z}}\\
&=z^3+10z^4+70z^5+424z^6+2382z^7+12804z^8+66946z^9+343772z^{10}+\cdots.
\end{align*}
Similarly, we have
\begin{align*}
T_2(z)&=\frac{1-\sqrt{2\sqrt{2\sqrt{1-4z}-1}-1}}{2z}+\frac{z}{\sqrt{1-4z}^3}-\frac{1}{\sqrt{1-4z}\sqrt{2\sqrt{1-4z}-1}}\\
&=8z^5+160z^6+1978z^7+19508z^8+168608z^9+1337684z^{10}+10003422z^{11}+\cdots,
\end{align*}

Note that
$\psi_j^k(z)=\sum_{i\geq0}I^k\binom{k/2}{i}(-2)^i\psi_{j-1}^i(z)$
for all $j\geq1$ and
$\psi_1^k=\sum_{i\geq0}\binom{k/2}{i}(-4)^iz^i(1-w)^i$, where
$I^2=-1$. Thus, for all $j\geq2$,
$$\psi_j(z)=\sum_{i_1,\ldots,i_j\geq0}I^{1+3i_j+\cdots+3i_3+2i_2+2i_1}2^{i_j+\cdots+i_2+2i_1} \binom{1/2}{i_j}\binom{i_j/2}{i_{j-1}}\cdots\binom{i_2/2}{i_1}z^{i_1}(1-w)^{i_1}.$$
Hence,  for all $s\geq0$ and $j\geq2$,
$$[w^s]\frac{\psi_j(z)}{2z(1-w)}=\sum_{i_1,\ldots,i_j\geq0}I^{1+3i_j+\cdots+3i_3+2i_2+2i_1+2s}2^{i_j+\cdots+i_2+2i_1-1} \binom{1/2}{i_j}\binom{i_1-1}{s}z^{i_1-1}\prod_{k=2}^{j}\binom{i_k/2}{i_{k-1}}$$
and
$$[w^s]\frac{\psi_1(z)}{2z(1-w)}=\sum_{i_1\geq0}I^{2i_1+2s}2^{2i_1-1}\binom{1/2}{i_1}\binom{i_1-1}{s}z^{i_1-1}.$$
Hence, by \eqref{eqtt4}, we can write an explicit formula for the
generating function $T_t(z)$ in terms of multi sums.




\bibliographystyle{acm}
\bibliography{refs_tier}

\end{document}